\documentclass[11pt]{article}
\usepackage{mathrsfs}
\usepackage{amsthm}
\usepackage{amssymb}
\usepackage{enumerate}
\usepackage{amsmath}
\usepackage{graphicx}
\usepackage{amsfonts}
\usepackage{float}
\usepackage{cite}
\usepackage[text={160mm,220mm}]{geometry}
\usepackage[colorlinks,
            linkcolor=blue,
            anchorcolor=blue,
            citecolor=blue
            ]{hyperref}
\usepackage{xcolor}
\usepackage{comment}

\usepackage{titlesec}
\titleformat{\subsection}[hang]{\it}{\thesubsection}{4 pt}{ \it}[]

\newtheorem{theorem}{Theorem}[section]
\newtheorem{remark}[theorem]{Remark}
\newtheorem{lemma}[theorem]{Lemma}

\numberwithin{equation}{section}
\normalsize

\title{Moderate Deviations for the SSEP with a Slow Bond}
\author{Xiaofeng Xue \thanks{\textbf{E-mail}: xfxue@bjtu.edu.cn \textbf{Address}: School of Science, Beijing Jiaotong University, Beijing 100044, China.} and { Linjie Zhao} \thanks{\textbf{E-mail}: zhaolinjie@pku.edu.cn \textbf{Address}: School of Mathematical Sciences, Peking University, Beijing 100871, China.}}
\date{}

\begin{document}
	
\maketitle

\begin{abstract}
We consider the one dimensional symmetric simple exclusion process with a slow bond.  In this model, particles cross each bond at rate $N^2$, except one particular bond, the slow bond, where the rate is $N$. Above,  $N$ is the scaling parameter. This model has been considered in the context of hydrodynamic limits, fluctuations and large deviations. We investigate moderate deviations from hydrodynamics and obtain a moderate deviation principle.
\end{abstract}

\noindent {\bf Keywords:} exclusion process, slow bond, moderate deviation, exponential martingale.

\section{Introduction}\label{section one}

The symmetric simple exclusion process (SSEP) with a slow bond was introduced in \cite{francogn2013slowbond} by Franco, Gon{\c{c}}alves and Neumann to derive from microscopic systems PDEs with boundary conditions, which has become a popular topic recently \cite{franco2016scaling,franco2019hydrodynamic,baldasso2017exclusion}.  The process evolves on the discrete ring with $N$ sites, where $N$ is the scaling parameter. There is at most one particle per site. Particles cross each bond at rate $N^2$ except one particular bond, where the rate is $N$.

The hydrodynamic limit of the SSEP with a slow bond has been  well understood \cite{francogn2013slowbond,francogn2015slowbond}.   The hydrodynamic equation turns out to be the heat equation with Robin's boundary conditions:
\begin{equation}\label{hydroEqn}
\left\{\begin{array}{ll}\partial_{t} \rho\,(t, u)=\partial_{u}^{2} \rho\,(t, u), &\quad t>0, \,u \in \mathbb{T} \backslash\{0\}, \\ \partial_{u} \rho\left(t, 0^{+}\right)=\partial_{u} \rho\left(t, 0^{-}\right)=\rho\left(t, 0^{+}\right)-\rho\left(t, 0^{-}\right), &\quad t>0, \\ \rho(0, u)=\gamma(u), &\quad u \in \mathbb{T},\end{array}\right.
\end{equation}
where $\mathbb{T}$ is the continuous ring, $0^+$ and $0^-$ denote respectively the right limit and left limit at site $0$, and $\gamma (\cdot)$ is the initial density profile.
Then it is natural to consider the equilibrium fluctuations and large deviations from the hydrodynamic limit. Equilibrium fluctuations have been studied in \cite{francogn2013slowbondfluc} and large deviations in \cite{francon2017largedeviationslowbond} by Franco, Gon{\c{c}}alves and Neumann.

To better understand the SSEP with a slow bond, we consider the moderate deviations from the hydrodynamic limit,  which gives asymptotic behavior of the model between the central limit theorem and the large deviation.  As far as we know, the only paper concerned about moderate deviations from hydrodynamics is \cite{gao2003moderate} authored by Gao and  Quastel, where the classic SSEP was considered. For literatures about theories of moderate deviations, see References \cite{Borovkov1978, deAcosta1998, Dembo1997, Gao1996, Wang2016, WangR2015, Wu1995} and so on.

\,

Next, we introduce the SSEP with a slow bond and main results. The process evolves on $\mathbb{T}_{N}=\{0,1,\ldots,N-1\}$ the ring with $N$ sites, with the convention $N \equiv 0$. Therefore, the state space is $\{0,1\}^{\mathbb{T}_N}$. For each configuration $\eta \in \{0,1\}^{\mathbb{T}_N}$, $\eta (x)=1$ means site $x$ is occupied by a particle, and $\eta (x) = 0$ means site $x$ is vacant. The infinitesimal generator $\mathcal{L}_N$ of the process is
\begin{equation*}
\mathcal{L}_N f (\eta) = N [f (\eta^{-1,0}) - f (\eta)] + N^2 \sum_{x \in \mathbb{T}_N, \atop x \neq -1} [f (\eta^{x,x+1}) - f(\eta)],
\end{equation*}
where
\[
\eta^{x,y}(u)=
\begin{cases}
\eta(u) & \text{~if~}u\neq x, y,\\
\eta(y) & \text{~if~}u=x,\\
\eta(x) & \text{~if~}u=y
\end{cases}
\]
for any $x\neq y$.   Denote by $\{\eta_t\}_{t\geq 0}$  the process with generator $\mathcal{L}_N$.  We suppress the dependence of the process $\{\eta_t\}_{t\geq 0}$ on $N$ for short.

Equivalently, we can define the process in the following way.  For each $i\neq -1$, let $\{Y_i(t)\}_{t\geq 0}$ be a Poisson process with rate $N^2$ and  $\{Y_{-1}(t)\}_{t\geq 0}$ be a Poisson process with rate $N$. Assume that all these Poisson processes are independent.  Then at any event moment of $Y_i (\cdot)$, $\eta(i)$ and $\eta(i+1)$ exchange their values.

The SSEP with a slow bond has a family of invariant measures indexed by the particle density. To be precise, let $\nu_\rho,\,\rho \in [0,1],$ be the product measure on $\mathbb{T}_N$ with marginals given by
$$
\nu_\rho \{\eta: \eta (x) = 1\} = \rho, \,\,\forall x \in \mathbb{T}_N.
$$
 Then, it can be checked easily that $\nu_\rho, \, \rho \in [0,1],$ are  reversible measures for the process $\{\eta_t\}_{t\geq 0}$.

\,

To define the empirical density and rate functions, we need to  introduce some definitions and notations and then discuss some topological issues. We identify $\mathbb{T}$ with $[0,1)$,  and thus $0^+$ with $0$ and $0^-$ with $1$.
 By the boundary conditions imposed on the hydrodynamic equation \eqref{hydroEqn}, it is reasonable to consider  test functions $G\in C^1[0,1]$ with the property
\begin{equation}\label{equ special function}
G^\prime (0) = G^\prime (1) = G (0) - G (1).
\end{equation}
The result of this paper  relies heavily on the  above kind of functions, especially trigonometric functions satisfying \eqref{equ special function}.
Define $\mathscr{G}_0$ as
\[
\mathscr{G}_0 := {\rm span}\left(\left\{\sin\left(k_n\left(x-\frac{1}{2}\right)\right)\right\}_{n\geq 1}\,\bigcup\, \left\{\cos\big(2n\pi x\big)\right\}_{n\geq 0}\right),
\]
where $k_n$ is the unique solution to the equation $-\frac{x}{2}=\tan\frac{x}{2}$ in ${\big((2n-1)\pi, (2n+1)\pi\big)}$ for each $n \geq 1$. It can be checked easily that  any $G \in \mathscr{G}_0$ satisfies \eqref{equ special function}. According to  \cite[Theorem 1]{Franco2009moderate} given by Franco and Landim, we can prove the set of the above trigonometric functions is a basis in $L^2 [0,1]$, which is crucial  to construct the topology of this paper.

\begin{lemma}\label{lemma 1.1.1 topology}
The set
	$\left\{\sin\big(k_n(x-1/2)\big)\right\}_{n\geq 1}\,\bigcup \,\left\{\cos\big(2n\pi x\big)\right\}_{n\geq 0}$ is an orthogonal basis of $L^2[0,1]$.
\end{lemma}

We put the proof of Lemma \ref{lemma 1.1.1 topology} in the appendix.

\,

Let $\mathscr{M}$ be  the space of linear (not necessarily bounded) functionals on $\mathscr{G}_0$ endowed with the following topology: for any $\mathscr{A}_n \in \mathscr{M},\, n \geq 1,$ and $\mathscr{A} \in \mathscr{M}$,
\[
\lim_{n\rightarrow+\infty}\mathscr{A}_n=\mathscr{A} \quad  \text{in $\mathscr{M}$}  \qquad \text{if and only if}  \qquad \lim_{n\rightarrow+\infty}\mathscr{A}_n(\theta_k)=\mathscr{A}(\theta_k) \quad \text{for all integers $k$},
\]
where $\theta_n(x)=\sin\big(k_n(x-1/2)\big)$ for $n\geq 1$ and $\theta_{-n}(x)=\cos\big(2n\pi x\big)$ for $n\geq 0$.  The above topology is metrizable and the metric $d\,(\cdot,\cdot)$ is given by
\[
d\big(\mathscr{A}_1, \mathscr{A}_2\big)=\sum_{-\infty<n<+\infty}\frac{1}{2^{|n|}}\frac{|\mathscr{A}_1(\theta_n)-\mathscr{A}_2(\theta_n)|}{1+|\mathscr{A}_1(\theta_n)-\mathscr{A}_2(\theta_n)|},\quad
\mathscr{A}_1, \,\mathscr{A}_2\in \mathscr{M}.
\]
It can be checked  directly that the space $\mathscr{M}$ is complete and separable under the above metric. Note that a bounded signed measure $\mu$ on $[0,1]$ can be identified with an element in $\mathscr{M}$ in the sense that $\mu(f)=\int_{[0,1]}f(x)\mu(dx)$ for any $f\in \mathscr{G}_0$.

\,

\begin{remark}\label{remark 1.2}
We construct the above topology for  technical reasons. Mainly, we can not show the uniqueness or existence of the weak solution to a PDE arising from hydrodynamics of the SSEP with a slow bond under a Girsanov's transformed measure. However, if we do not distinguish two measures $\mu_1$ and $\mu_2$ satisfying $\mu_1(\theta_n)=\mu_2(\theta_n)$ for all $n$, the above PDE can be reduced to an ODE on $\mathscr{M}$, the existence and uniqueness of the solution to which can be rigorously proved. For mathematical details, see Section \ref{section of lower bound} and appendix.
\end{remark}

\,

In the following, we will fix a horizonal time $T > 0$. Let $D\big([0,T],  \mathscr{M} \big)$ be the  space of c\`{a}dl\`{a}g functions from $[0,T]$ to $\mathscr{M}$ endowed with the Skorohod topology. Define the rescaled central empirical density $\mu^N_t (d u)$ as
\begin{equation*}
\mu^N_t (d u) := \frac{1}{a_N} \sum_{x \in \mathbb{T}_N} (\eta_t (x) - \rho) \delta_{x/N} (d u),
\end{equation*}
where $\sqrt{N} \ll a_N \ll N$.  We will regard  $\mu^N :=\{\mu_t^N\}_{0\leq t\leq T}$ as  a random element taking values in $D\big([0,T], \mathscr{M} \big)$.

Let $\mathscr{G}$ be the family of functions $G :[0,T] \times [0,1] \rightarrow \mathbb{R}$ with the following forms: there exist $M \in \mathbb{N}$ and $b_m (t) \in C^1 ([0,T])$, $- M \leq m \leq M$ such that
\begin{equation*}
G (t,u) = \sum_{m=-M}^M b_m (t)\, \theta_m (u),  \quad (t,u) \in [0,T] \times [0,1].
\end{equation*}
Then for any $G \in \mathscr{G}$,
\begin{equation}\label{bc}
\partial_u G (t,0) = \partial_u G (t,1) = G (t,0) - G (t,1),\quad \forall\, t \in [0,T].
\end{equation}
We sometimes write $G_t (u)$ for $G (t,u)$. For $G \in \mathscr{G}$, define the extended Laplacian $\tilde{\Delta}$ as
\[
\tilde{\Delta} G_t \,(u) = \begin{cases}
\partial_u^2 G_t\, (u) &\text{if $u \neq 0$},\\
\partial_u^2 G_t\, (0^+) &\text{if $u = 0$}.
\end{cases}
\]

Fix a density $\rho \in (0,1)$.    Denote by $Q^N_\rho$  the law of  $\{\mu^N_t\}_{0 \leq t \leq T}$  with initial distribution  $\nu_\rho$. Let $\mathbb{P}^N_\rho$ be the law of the process $\{\eta_t\}_{0 \leq t \leq T}$ with initial distribution $\nu_\rho$, and $\mathbb{E}^N_\rho$ the corresponding expectation. Let $\rm{E}_{\nu_\rho}$ be the expectation with respect to $\nu_\rho$. For $\mu \in D ([0,T],\mathscr{M})$, define
\begin{equation}
\begin{aligned}
&I (\mu) := I_{ini} (\mu_0) + I_{dyn} (\mu),\\
&I_{ini}(\mu_0) :=\sup_{\gamma \in \mathscr{G}_0} \left\{\mu_0 (\gamma) -\frac{\rho(1-\rho)}{2}\int_0^1 \gamma^2(u) \,du \right\},\\
&I_{dyn} (\mu) := \sup_{G \in \mathscr{G}} \left\{ \ell_T (\mu,G) - \rho (1-\rho) \int_0^T \left(G_t (0) - G_t (1) \right)^2 dt  \right. \\
&\left. \qquad \qquad \qquad \qquad \qquad \qquad
- \rho (1-\rho) \int_0^T \int_0^1 \left( \partial_u G_t (u) \right)^2 \,d u \,d t\right\},
\end{aligned}
\end{equation}
where
\begin{align}\label{l_T}
\ell_T (\mu,G) := \mu_T\left(G_T\right)  -  \mu_0 \left(G_0\right)  - \int_0^T  \mu_t \left((\partial_t + \tilde{\Delta}) G_t\right)  \,d t.
\end{align}

Now we are ready to state  the main result of the paper.

\,

\begin{theorem}\label{t1 main}
For any closed set $C$ of $D \left( [0,T], \mathscr{M} \right) $,
\begin{equation}\label{eqnupperbound}
\limsup_{N \rightarrow \infty} \frac{N}{a_N^2} \log Q^N_\rho\, [C]  \leq - \inf_{\mu \in C} I (\mu),
\end{equation}
and for any open set $O$ of $D \left( [0,T], \mathscr{M}  \right)$,
\begin{equation}\label{equ lower bound}
\liminf_{N \rightarrow \infty} \frac{N}{a_N^2} \log Q^N_\rho\, [O]  \geq - \inf_{\mu \in O} I (\mu).
\end{equation}
\end{theorem}

\,

\begin{remark}\label{remark 1.4}

We recall the large deviation principle of the SSEP with a slow bond established in \cite{francon2017largedeviationslowbond} by Franco and Neumann for a comparison. Note that definitions and notations in this remark are not utilized elsewhere.  Let
\[
\pi^N_t (d u)= \frac{1}{N} \sum_{x \in \mathbb{T}_N} \eta_t (x) \delta_{x/N} (d u),
\quad
\pi^N=\{\pi^N_t\}_{0\leq t\leq T},
\]
then it was shown in \cite{francon2017largedeviationslowbond} that, roughly speaking,
\[
P(\pi^N\approx\pi)\approx \exp\left\{ -N J(\pi) \right\}
\]
assuming uniqueness for the weak solution to hydrodynamic equation associated to the perturbed process,  where
\begin{align*}
J(\pi)=\sup_{H\in C^{1,2}\left([0, T]\times [0,1]\right)}\{\widehat{\ell}_H(\pi)-\Phi_H(\pi)\}
\end{align*}
with $\widehat{\ell}_H(\pi)$ given by
\begin{align*}
\widehat{\ell}_H(\pi)=&\langle \rho_T, H_T\rangle-\langle\rho_0, H_0\rangle-\int_0^T\langle \rho_t, \left(\partial_t+\Delta \right)H_t\rangle \,dt\\
&-\int_0^T \left\{\rho_t(0)\partial_uH_t(0)-\rho_t(1)\partial_uH_t(1)\right\}\,dt+\int_0^T \left(\rho_t(0)-\rho_t(1)\right)\delta H_t(0)\,dt
\end{align*}
and $\Phi_H(\pi)$ given by
\begin{align*}
\Phi_H(\pi)=&\int_0^T\langle\chi(\rho_t), \left(\partial_uH_t\right)^2\rangle\, dt+\int_0^T\rho_t(1)\left(1-\rho_t(0)\right)\psi\left(\delta H_t(0)\right)\,dt\\
&+\int_0^T\rho_t(0)\left(1-\rho_t(1)\right)\psi\left(-\delta H_t(0)\right)\,dt,
\end{align*}
where $\rho_t$ is the Radon-Nikodym  derivative of $\pi_t$ with respect to the Lebesgue measure, $\psi(x)=e^x-x-1$, $\chi(\rho)=\rho(1-\rho)$ and $\delta H_t(0)=H_t(0)-H_t(1)$. Since $e^x-x-1=\frac{x^2}{2}+o(x^2)$ as $|x|$ decreases to $0$ and
$\widehat{\ell}_H(\pi)$ equals $\ell_T(\pi, H)$ when $H$ satisfies \eqref{bc},  the rate function $I_{dyn}$ can be intuitively considered as the quadratic part of $J$ about its minimum, which is a common relationship between large and moderate deviations for many models in statistical physics.
\end{remark}

\,

{\bf Notation.}  For deterministic positive sequences $\{b_n\}_{n\geq 1}$, $\{c_n\}_{n\geq 1}$ and random sequence $\{X_n\}_{n\geq 1}$, we write $b_n=o(c_n)$ if $
\limsup_{n \rightarrow \infty}\, b_n/c_n = 0
$ and $b_n = \mathcal{O} (c_n)$ if
$
\limsup_{n \rightarrow \infty}\, b_n/c_n  < C
$
for some constant $C$ independent of $n$. We also write $b_n = \mathcal{O}_G (c_n)$ to stress the dependence on some parameter $G$ of the constant $C$.
We write  $X_n=o_p(c_n)$ if $X_n / c_n \rightarrow 0$ in probability as $n \rightarrow \infty$, and  $X_n=o_{\exp}(c_n)$ if
\[
\limsup_{n\rightarrow+\infty}\,\frac{1}{c_n}\,\log P\big(|X_n|>\epsilon\big)=-\infty, \quad \forall \epsilon > 0.
\]
We  remark on these last points that the constant throughout the paper may be different from line to line.

\,

The rest of the paper is devoted to the proof of Theorem \ref{t1 main}. In Section  \ref{secsed} we give several super-exponential  estimates that are necessary in the proof of upper and lower bounds as a preparation.  Moderate upper bounds are proved in Section \ref{sec upper}. Our proof follows a strategy similar with that introduced in \cite{gao2003moderate}, except for some details modified due to  technical reasons caused by the slow bond. First, as introduced above,  we have to choose a proper topology and  to consider the empirical density as a random element taking values in the linear functional space $\mathscr{M}$,  instead of the dual  of Schwartz  functions.  Second,  an extra super-exponential estimate  (Lemma \ref{lem1}) is needed. Third, because of the topology constructed, we have to use a different version of Minimax Theorem (Theorem \ref{thm minimax}) from the one in \cite{gao2003moderate}. Moderate lower bounds are proved in Section \ref{section of lower bound}. A crucial step in the proof is the utilizing of a generalized Girsanov's theorem to give the hydrodynamic equation of the model under a transformed measure.

\section{Super-exponential Decay}\label{secsed}

In this section, we mainly present three super-exponential estimates that are critical when making some replacements and proving exponential tightness.

\,

\begin{lemma}\label{lem1}
For any continuous  function $G: [0,T] \rightarrow \mathbb{R}$ and any $\delta,\,t > 0$,
	\begin{equation}\label{l2}
	\limsup_{N \rightarrow \infty} \frac{1}{a_N} \log \mathbb{P}^N_\rho \left[ \left|\int_0^t \left(\eta_s(0) (1-\eta_s(-1)) - \rho (1-\rho)\right) G_s \, ds\right| > \delta\right] = - \infty.
	\end{equation}
The same result holds with $\eta_s (0) (1-\eta_s (-1))$ replaced by $\eta_s (-1) (1-\eta_s (0))$.
\end{lemma}

\begin{proof}
We only present the proof of  \eqref{l2} since the rest is the same. For any integer $M > 0$ and $x \in \mathbb{T}_N$, define  $\eta^{M,\rm{R}} (x)$ (resp. $\eta^{M,\rm{L}} (x)$) as the average density over the box of size $M$ to the right (resp. left) of site $x$ ,
	\begin{equation*}
	\eta^{M,\rm{R}} (x) = \frac{1}{M} \sum_{y=x}^{x+M-1} \eta(y), \quad
	\eta^{M,\rm{L}} (x) = \frac{1}{M} \sum_{y=x-M+1}^{x} \eta(y).
	\end{equation*}
	Note that  for every integer $M > 0$,
	\begin{align*}
	\eta (0) (1-\eta (-1)) &- \rho (1-\rho) =
	\left(\eta(0) - \eta^{M,\rm{R}} (0)  \right) \left(1-\eta(-1)\right) \\
	&+ \eta^{M,\rm{R}} (0)  (\eta^{M,\rm{L}} (-1)  - \eta(-1))
	+ \left(\eta^{M,\rm{R}} (0) - \rho \right)\left(1-\eta^{M,\rm{L}} (-1) \right)\\
	&+ \rho \left(\rho - \eta^{M,\rm{L}} (-1) \right).
	\end{align*}
	Since for any positive sequences $\{b_N\}_{N \geq 1}$ and $\{c_N\}_{N \geq 1}$,
	$$
	\limsup_{N \rightarrow \infty} \frac{1}{a_N} \log (b_N+c_N) \leq \max \left\{  	\limsup_{N \rightarrow \infty} \frac{1}{a_N} \log b_N,\,	\limsup_{N \rightarrow \infty} \frac{1}{a_N} \log c_N \right\},
	$$
	to prove \eqref{l2}, we only need to prove for any $\delta > 0$,
	\begin{equation}\label{a1}
	\limsup_{N \rightarrow \infty} \frac{1}{a_N} \log \mathbb{P}^N_\rho \left[ \left|\int_0^t \left(\eta_s(0) - \eta_s^{M,\rm{R}} (0)  \right) \left(1-\eta_s(-1)\right) G_s \,ds\right| > \delta\right] = - \infty,
	\end{equation}
	\begin{equation*}
	\limsup_{N \rightarrow \infty} \frac{1}{a_N} \log \mathbb{P}^N_\rho \left[ \left|\int_0^t \eta_s^{M,\rm{R}} (0)  (\eta_s^{M,\rm{L}} (-1)  - \eta_s(-1)) G_s \,ds\right| > \delta\right] = - \infty,
	\end{equation*}
	\begin{equation}\label{a2}
	\limsup_{N \rightarrow \infty} \frac{1}{a_N} \log \mathbb{P}^N_\rho \left[ \left|\int_0^t \left(\eta_s^{M,\rm{R}} (0) - \rho \right)\left(1-\eta_s^{M,\rm{L}} (-1) \right) G_s \,ds\right| > \delta\right] = - \infty,
	\end{equation}
	and
	\begin{equation*}
	\limsup_{N \rightarrow \infty} \frac{1}{a_N} \log \mathbb{P}^N_\rho \left[ \left|\int_0^t \rho \left(\rho - \eta_s^{M,\rm{L}} (-1) \right) G_s \,ds\right| > \delta\right] = - \infty.
	\end{equation*}
	We only prove \eqref{a1} and \eqref{a2}, since the remaining two terms are similar.
	
	\,
	
	For any $A > 0$,  by  Chebyshev's inequality, the formula on the left-hand side of \eqref{a1} is bounded from above by
	\begin{equation}\label{a3}
	-\frac{A \delta}{a_N} + \frac{1}{a_N} \log \mathbb{E}^N_\rho \left[ \exp \left\{ A   \left|\int_0^t \left(\eta_s(0) - \eta_s^{M,\rm{R}} (0)  \right) \left(1-\eta_s(-1)\right) G_s \,ds\right|\right\} \right].
	\end{equation}
	Since $e^{|x|} \leq e^x + e^{-x}$, we can remove the modulus in the expectation above. By the Feynman-Kac formula  (see \cite[Lemma A.1.7.2]{klscaling} by Kipnis and Landim for example), the second term in \eqref{a3} is bounded by
	\begin{equation*}
	\frac{1}{a_N} \int_0^t d s \sup_{\text{$f$ density}} \left\{ A G_s \int \left(\eta (0) - \eta^{M,\rm{R}} (0)  \right) \left(1-\eta(-1)\right) f (\eta) \,d \nu_\rho -   \mathcal{D}_N \left( f; \nu_\rho \right) \right\},
	\end{equation*}
	where $\mathcal{D}_N \left( f; \nu_\rho \right)$ is the Dirichlet form of $f$ associated with $\nu_\rho$ given by
	\begin{equation*}
	\begin{aligned}
	\mathcal{D}_N \left( f; \nu_\rho \right) &:= \left\langle \sqrt{f},  (- \mathcal{L}_N) \sqrt{f}\, \right\rangle_{\nu_\rho} \\
	& =  N \left[\sqrt{f (\eta^{-1,0})} - \sqrt{f (\eta)}\right]^2
	+ N^2 \sum_{x \in \mathbb{T}_N, \atop x \neq -1} \left[\sqrt{f (\eta^{x,x+1})} - \sqrt{f(\eta)}\right]^2.
	\end{aligned}
	\end{equation*}
	We first  write $\eta (0) - \eta^{M,\rm{R}} (0)$ as a telescope sum,
	\begin{equation*}
	\eta (0) - \eta^{M,\rm{R}} (0) = \frac{1}{M} \sum_{x=0}^{M-1} \sum_{y=0}^{x-1} \left( \eta(y) - \eta(y+1) \right).
	\end{equation*}
	Making the transformations $\eta \rightarrow \eta^{y,y+1}$,  by Cauchy-Schwartz inequality, we obtain that there exists a constant $C$ only depending on $G$ such that for any $B > 0$,
	\begin{align*}
	&A G_s \int \left(\eta (0) - \eta^{M,\rm{R}} (0)  \right) \left(1-\eta(-1)\right) f (\eta) \,d \nu_\rho \\
	& = \frac{A G_s}{2 M}    \sum_{x=0}^{M-1} \sum_{y=0}^{x-1}  \int \left( \eta(y) - \eta(y+1) \right)
	\left(1-\eta(-1)\right) \left(f(\eta) - f(\eta^{y,y+1})\right) d \nu_\rho\\
	& \leq \frac{A B ||G||_{\infty}}{4 M}    \sum_{x=0}^{M-1} \sum_{y=0}^{x-1} \int  \left(\sqrt{f(\eta)} - \sqrt{f(\eta^{y,y+1})}\right)^2 d \nu_\rho \\
	&\quad \quad \quad \quad \quad \quad + \frac{A ||G||_{\infty}}{4 B M}    \sum_{x=0}^{M-1} \sum_{y=0}^{x-1} \int  \left(\sqrt{f(\eta)} + \sqrt{f(\eta^{y,y+1})}\right)^2 d \nu_\rho\\
	&\leq  C   \left( \frac{A B }{ N^2}  \mathcal{D}_N \left( f; \nu_\rho \right) + \frac{A M}{B} \right).
	\end{align*}
	Taking $B = N^2 A^{-1} C^{-1}$, we bound \eqref{a3} by
	\begin{equation*}
	\inf_{A > 0} \,\left\{-\frac{A \delta}{a_N}  + \frac{ A^2 C^2 t  M}{N^2 a_N}\right\} = - \frac{\delta^2 N^2}{4 C^2 t M  a_N}.
	\end{equation*}
	We prove \eqref{a1} by  choosing $M$ such that $M a_N \ll N^2$.
	
	\,
	
	As above, for any $A > 0$, the formula on the left-hand side of \eqref{a2} is bounded  by
	\begin{align*}
	-\frac{A \delta}{a_N} + \frac{1}{a_N} \log \mathbb{E}^N_\rho \left[ \exp \left\{ A   \left| \int_0^t \left(\eta_s^{M,\rm{R}} (0) - \rho \right)\left(1-\eta_s^{M,\rm{L}} (-1) \right) G_s \,ds \right|\right\} \right].
	\end{align*}
	As before, we can first remove the modulus.  By Jensen's inequality and the invariance of the measure $\nu_\rho$, we bound the above formula by
	\begin{equation}\label{a4}
	-\frac{A \delta}{a_N} + \frac{1}{a_N} \log \left(\frac{1}{t} \int_0^t d s \,\rm{E}_{\nu_\rho} \left[ \exp \left\{ A  t G_s  \left(\eta^{M,\rm{R}} (0) - \rho \right)\left(1-\eta^{M,\rm{L}} (-1) \right)  \right\} \right] \right).
	\end{equation}
By Taylor's expansion, the expectation in the above formula is  less than or equal to
\begin{equation*}
\begin{aligned}
 &\sum_{k \geq 0} \frac{A^{2k} t^{2k} ||G||_\infty^{2k}}{(2k)!}  \,\rm{E}_{\nu_\rho} \left[ \left(\eta^{M,\rm{R}} (0) - \rho \right)^{2k}\right] \\
 &\quad \quad + \sum_{k \geq 0} \frac{A^{2k+1} t^{2k+1} ||G||_\infty^{2k+1}}{(2k+1)!}  \,\rm{E}_{\nu_\rho} \left[ \left|\eta^{M,\rm{R}} (0) - \rho \right|^{2k+1}\right] \\
 & \leq (1+ At||G||_\infty)\sum_{k \geq 0} \frac{A^{2k} t^{2k} ||G||_\infty^{2k}}{(2k)!}  \,\rm{E}_{\nu_\rho} \left[ \left(\eta^{M,\rm{R}} (0) - \rho \right)^{2k}\right].
 \end{aligned}
\end{equation*}
Some computations (see the appendix) show that there exists a constant $ C (\rho)$  such that
\begin{equation}\label{equ control of 2k moments}
 {\rm E}_{\nu_\rho} \left[ \left(\eta^{M,{\rm R}} (0) - \rho \right)^{2k} \right] \leq \frac{C(\rho)^k\,  k!}{M^k}.
\end{equation}
Since $2^k( k! )^2 \leq (2k)!$, the expectation  in \eqref{a4} is bounded by $C A \exp\{C A^2/M\}$ for some constant $C = C (t,G,\rho)$. Therefore, we bound \eqref{a4} by
\begin{equation*}
-\frac{A \delta}{a_N} + \frac{C  A^2}{M a_N} + \frac{\log C + \log A}{a_N}.
\end{equation*}
We finish the proof by taking $M = [N/2]$ and $A=  M \delta / (2 C)$.
\end{proof}

\,

\begin{lemma} \label{lem2}
For any $G \in C \left([0,T]\times [0,1]\right)$ and any $\delta,\,t > 0$,
	\begin{equation*}
	\begin{aligned}
	\limsup_{N \rightarrow \infty} \frac{1}{a_N} \log \mathbb{P}^N_\rho
	 \left[\left|\int_0^t  \frac{ 1}{N} \sum_{x \in \mathbb{T}_N,\atop x \neq -1}  \left[ (\eta_s (x) - \eta_s (x+1) )^2- 2 \rho(1-\rho) \right] G_s \left(\frac{x}{N}\right) \,d s \right| > \delta\right]
	= - \infty.
	\end{aligned}
	\end{equation*}
\end{lemma}

\,

\begin{lemma}\label{lem3}
Let $G \in C [0,1]$. Then for any $t > 0$,
	\begin{equation}
	\limsup _{A \rightarrow \infty}\, \limsup _{N \rightarrow \infty} \,\frac{N}{a_N^2} \log \mathbb{P}^{N}_{\rho}  \left[ \sup_{0 \leqslant t \leqslant T} \left| \int_0^t \langle \mu^N_s, G\rangle \,d s \right| > A \right] = -\infty,
	\end{equation}
and for any $\epsilon,\,t > 0$,
\begin{equation}
\limsup_{\delta \rightarrow 0} \,\limsup_{N \rightarrow \infty}\, \frac{N}{a_N^2} \log \mathbb{P}^N_\rho \left[\sup_{|t-s| \leq \delta}  \left| \int_s^t \left\langle \mu^N_u, G \right\rangle \,d u   \right| > \epsilon \right] = - \infty.
\end{equation}
\end{lemma}

\,

The  proof  of  \cite[Lemmas 2.1 and 2.2]{gao2003moderate} also applies to the above two lemmas.  The main ingredients are the invariance of the Bernoulli product measure $\nu_\rho$. For that reason we omit the proof.

\section{Upper Bound}\label{sec upper}

In this section, we prove \eqref{eqnupperbound} the moderate deviations upper bound.  The strategy is first proving upper bound over compact sets, and then extending to closed sets, which follows from the exponential tightness.

\,

Fix $G \in {\mathscr{G}}$.  By Feynman-Kac formula (see \cite[A.1.7]{klscaling}),
\begin{equation}\label{martingale1}
M_t^N (G):= \frac{f(t,\eta_t)}{f(0,\eta_0)} \exp \left\{ - \int_0^t \frac{\partial_s f + \mathcal{L}_N f }{f} (s,\eta_s) \,d s\right\}
\end{equation}
is a positive mean-one martingale, where
\begin{equation*}
f (t,\eta) := f_G (t,\eta) :=   \exp \left\{ \frac{a_N}{N} \sum_{x \in \mathbb{T}_N} \left(\eta (x) - \rho\right) G_t \left(x/N\right)\right\}.
\end{equation*}
Notice that
$$
f (t,\eta_t) = \exp \left\{ \frac{a_N^2}{N} \left\langle \mu^N_t,G_t  \right\rangle \right\}.
$$
A simple calculation yields that
\begin{align*}
(\partial_s f + \mathcal{L}_N f) (s,\eta_s) &= f (s,\eta_s)  \Bigg( \frac{a_N^2}{N} \langle \mu^N_s,\partial_s G_s \rangle \\
&+ N \left[ \exp \left\{\frac{a_N}{N} (\eta_s (-1) - \eta_s (0)) \left( G_s \left(\frac{0}{N}\right) - G_s \left(\frac{-1}{N}\right) \right)\right\} -1 \right]\\
&\left. + \sum_{x \in \mathbb{T}_N,\atop x \neq -1} N^2 \left[ \exp \left\{\frac{a_N}{N} (\eta_s (x) - \eta_s (x+1)) \left( G_s \left(\frac{x+1}{N}\right) - G_s \left(\frac{x}{N}\right) \right)\right\} -1 \right]  \right).
\end{align*}
By Taylor's expansion,
\begin{align*}
&N \left[ \exp \left\{\frac{a_N}{N} (\eta_s (-1) - \eta_s (0)) \left( G_s \left(\frac{0}{N}\right) - G_s \left(\frac{-1}{N}\right) \right)\right\} -1 \right]\\
&= a_N (\eta_s (-1) - \eta_s (0)) \left( G_s \left(\frac{0}{N}\right) - G_s \left(\frac{-1}{N}\right) \right) \\
&+ \frac{a_N^2}{2 N}  (\eta_s (-1) - \eta_s (0))^2 \left( G_s \left(\frac{0}{N}\right) - G_s \left(\frac{-1}{N}\right) \right)^2 + \mathcal{O}_G \left(\frac{a_N^3}{N^2}\right),
\end{align*}
and for $x \neq -1$,
\begin{align*}
&N^2 \left[ \exp \left\{\frac{a_N}{N} (\eta_s (x) - \eta_s (x+1)) \left( G_s \left(\frac{x+1}{N}\right) - G_s \left(\frac{x}{N}\right) \right)\right\} -1\right]\\
&= N a_N  (\eta_s (x) - \eta_s (x+1)) \left( G_s \left(\frac{x+1}{N}\right) - G_s \left(\frac{x}{N}\right) \right)\\
&+ \frac{a_N^2}{2} (\eta_s(x) - \eta_s(x+1))^2 \left( G_s \left(\frac{x+1}{N}\right) - G_s \left(\frac{x}{N}\right)^2 \right) + \mathcal{O}_G \left(\frac{ a_N^3}{N^4}\right).
\end{align*}
Using the summation by parts formula,
\begin{align}
M_t^N (G) &= \exp \frac{a_N^2}{N} \bigg\{ \langle \mu^N_t, G_t\rangle - \langle \mu^N_0, G_0\rangle  - \int_0^t \langle \mu^N_s, (\partial_s + \tilde{\Delta}) G_s\rangle ds \label{r0}\\
&- \int_0^t \frac{N}{a_N} (\eta_s (-1) - \eta_s (0)) \left( G_s \left(\frac{0}{N}\right) - G_s \left(\frac{-1}{N}\right) \right) d s\label{r1} \\
&- \int_0^t \frac{1}{2}  (\eta_s (-1) - \eta_s (0))^2 \left( G_s \left(\frac{0}{N}\right) - G_s \left(\frac{-1}{N}\right) \right)^2 ds \label{r2}\\
&- \int_0^t \frac{N}{a_N} \left( (\eta_s (0) - \rho) \nabla_N G_s \left(\frac{0}{N}\right) - (\eta_s (-1) - \rho) \nabla_N G_s \left(\frac{-2}{N}\right)\right) ds \label{r3}\\
&- \int_0^t \frac{1}{2 N} \sum_{x \in \mathbb{T}_N,\atop x \neq -1}   (\eta_s (x) - \eta_s (x+1))^2 \left(\nabla_N G_s \left(\frac{x}{N}\right)\right)^2 d s \label{r4}\\
&+ \mathcal{O}_G \left(\frac{a_N}{N}\right) + \mathcal{O}_G \left(\frac{ 1}{a_N}\right) \bigg\},\label{r5}
\end{align}
where $\nabla_N$ is the discrete space derivative, $\nabla_N G_s (x/N) := N \left[ G_s \left((x+1)/N\right) - G_s (x/N) \right]$. By the boundary condition \eqref{bc} imposed on $G$, the sum of \eqref{r1} and \eqref{r3} is of order $\mathcal{O}_G \left( a_N^{-1} \right)$. Therefore,
\begin{equation*}
\begin{aligned}
M_t^N (G) &= \exp \frac{a_N^2}{N} \left\{ \ell_t (\mu^N) -  \int_0^t \frac{1}{2}  (\eta_s (-1) - \eta_s (0))^2 \left( G_s \left(\frac{0}{N}\right) - G_s \left(\frac{-1}{N}\right) \right)^2 ds\right.\\
&\left.- \int_0^t \frac{1}{2 N} \sum_{x \in \mathbb{T}_N,\atop x \neq -1}   (\eta_s (x) - \eta_s (x+1))^2 \left(\nabla_N G_s \left(\frac{x}{N}\right)\right)^2 d s + \mathcal{O}_G \left(\frac{ 1}{a_N}\right)  \right\}.
\end{aligned}
\end{equation*}

\,

\begin{lemma}[Upper bounds over compact sets]
For any compact  set $K \subset D \left( [0,T], \mathscr{M} \right) $,
\begin{equation}
\limsup_{N \rightarrow \infty} \, \frac{N}{a_N^2} \log Q^N_\rho [K]  \leq - \inf_{\mu \in K} I(\mu).
\end{equation}
\end{lemma}

\begin{proof}
For any $\delta > 0$ and any $G \in {\mathscr{G}}$, let
\begin{equation*}
\begin{aligned}
B_{N,\delta} &=  \left\{ \left|   \int_0^T \sum_{x \in \mathbb{T}_N,\atop x \neq -1}  \frac{1}{2 N} (\eta_t (x) - \eta_t (x+1))^2 \left(\nabla_N G_t \left(\frac{x}{N}\right)\right)^2 d t - \int_0^T \int_0^1 \rho (1-\rho) \left(\nabla G_t (u)\right)^2 \,dt \,du\right| <  \delta \right\}\\
&\bigcap \left\{ \left| \int_0^T \left[ \frac{1}{2}  (\eta_t (-1) - \eta_t (0))^2 - \rho (1-\rho) \right] \left( G_t \left(\frac{0}{N}\right) - G_t \left(\frac{-1}{N}\right) \right)^2 d t \right| < \delta  \right\}.
\end{aligned}
\end{equation*}
By Lemmas \ref{lem1},\,\ref{lem2} and the assumption $a_N \ll N$,
\begin{equation*}
\limsup_{N \rightarrow \infty} \frac{N}{a_N^2} \log \mathbb{P}^N_\rho [B_{N,\delta}^c] = - \infty.
\end{equation*}
Therefore, for any $\gamma \in \mathscr{G}_0$,
\begin{equation*}
\begin{aligned}
\limsup_{N \rightarrow \infty} \,\frac{N}{a_N^2} \log \mathbb{P}^N_\rho [\mu^N \in K]
& \leq \limsup_{N \rightarrow \infty} \frac{N}{a_N^2} \log \mathbb{P}^N_\rho \left[\left\{\mu^N \in K \right\} \cap B_{N,\delta}  \right]\\
&=  \limsup_{N \rightarrow \infty} \frac{N}{a_N^2} \log \mathbb{E}^N_\rho \left[ \left(M_T^N (G)\right)^{-1} M_T^N (G) \mathbf{1}_{ \left\{\mu^N \in K\right\} \bigcap B_{N,\delta} } \right]\\
&\leq \sup_{\mu \in K} \left\{- \ell_T (\mu)  + \rho (1-\rho) \int_0^T \left(G_t (0) - G_t (1) \right)^2 dt   \right.\\
&\qquad \qquad \left. + \rho (1-\rho) \int_0^T \int_0^1 \left( \partial_u G_t (u) \right)^2 \,d u \,d t - \mu_0 (\gamma ) \right\} + \mathcal{O} (\delta)\\
& + \limsup_{N \rightarrow \infty} \frac{N}{a_N^2} \log \mathbb{E}^N_\rho \left[  M_T^N (G) \exp \left\{ \frac{a_N^2}{N} \langle \mu^N_0,\gamma \rangle \right\}\right].
\end{aligned}
\end{equation*}
Because $\{M_t^N (G)\}$ is a mean one martingale and $\nu_\rho$ is a product measure, direct calculations yield that
\begin{equation*}
\limsup_{N \rightarrow \infty} \frac{N}{a_N^2} \log \mathbb{E}^N_\rho \left[  M_T^N (G) \exp \left\{ \frac{a_N^2}{N} \langle \mu^N_0,\gamma \rangle \right\}\right] = \frac{\rho (1-\rho)}{2} \int_0^1 \gamma (u)^2 \,d u.
\end{equation*}
Letting $\delta \rightarrow 0$,
and then minimizing over $G \in {\mathscr{G}}, \gamma \in \mathscr{G}_0$,
\begin{equation*}
\begin{aligned}
&\limsup_{N \rightarrow \infty} \frac{N}{a_N^2} \log \mathbb{P}^N_\rho [\mu^N \in K]
\leq \inf_{G \in {\mathscr{G}}, \atop \gamma \in \mathscr{G}_0} \,\sup_{\mu \in K}\,
\left\{ - \ell_T (\mu)  + \rho (1-\rho) \int_0^T \left(G_t (0) - G_t (1) \right)^2 dt \right.\\
&\qquad \left. + \rho (1-\rho) \int_0^T \int_0^1 \left( \partial_u G_t (u) \right)^2 \,d u \,d t -  \mu_0 (\gamma)  +  \frac{\rho (1-\rho)}{2} \int_0^1 \gamma (u)^2 \,d u \right\}.
\end{aligned}
\end{equation*}

\,

In order to exchange the supremum and infimum  above, we use the following  version of  Minimax Theorem proved by Nikaid{\^o}.

\,

\begin{theorem}[Minimax Theorem, {\cite[Theorem 1]{nikaido1953minimax}}]\label{thm minimax}
	Let $\mathbf{X}$ be a linear space endowed with separative topology and $\mathbf{Y}$ a linear space. Moreover, assume $\mathbf{X}$ is compact.  Let $f: \mathbf{X} \times \mathbf{Y} \rightarrow \mathbb{R}$ satisfy that $f (x,y)$ is convex   in $y$ for each fixed $x$, and   concave  in $x$ for each fixed $y$.  Furthermore,  $f (x,y)$ is continuous in $x$ for each fixed $y$. If $\sup_{x \in \mathbf{X}}  \, \inf_{y \in \mathbf{Y}} \, f (x,y)$ is finite, then
	\begin{equation*}
	\sup _{x \in \mathbf{X}} \, \inf _{y \in \mathbf{Y}} \, f(x, y)=\inf _{y \in \mathbf{Y}} \, \sup_{x \in \mathbf{X}} \, f(x, y).
	\end{equation*}
\end{theorem}

\,

We finish the proof by taking $\mathbf{X} = K \subset D \left( [0,T], \mathscr{M}\right),$ \,$\mathbf{Y} = \mathscr{G} \times \mathscr{G}_0$  and
\begin{align*}
f\left(\mu, (G, \gamma)\right)&= - \ell_T (\mu)  + \rho (1-\rho) \int_0^T \left(G_t (0) - G_t (1) \right)^2 dt \\
&+\rho (1-\rho) \int_0^T \int_0^1 \left( \partial_u G_t (u) \right)^2 \,d u \,d t -  \mu_0 (\gamma)  +  \frac{\rho (1-\rho)}{2} \int_0^1 \gamma (u)^2 \,d u
\end{align*}
for any $\mu\in \mathbf{X}$ and $(G, \gamma)\in \mathbf{Y}$.
\end{proof}

\,

To extend the moderate deviations  upper bound to any closed set, it suffices to show the exponential tightness of the  sequence $\{Q_N\}_{ N \geq 1}$, which follows from the following Lemma as in \cite{gao2003moderate}.

\,

\begin{lemma}\label{exptight}
For any $G \in \mathscr{G}_0$,
\begin{equation}\label{b1}
\limsup _{A \rightarrow \infty} \, \limsup _{N \rightarrow \infty}\, \frac{N}{a_N^2} \log \mathbb{P}^{N}_{\rho}\left(\sup _{0 \leqslant t \leqslant T}\left|\left\langle\mu_{t}^{N}, G\right\rangle\right|>A\right)=-\infty,
\end{equation}
and  for any $\epsilon > 0$,
\begin{equation}\label{u0}
\limsup _{\delta \rightarrow 0} \, \limsup _{N \rightarrow \infty} \,\frac{N}{a_N^2} \log  \mathbb{P}^{N}_{\rho}\left(\sup _{0<t \leqslant \delta}\left|\left\langle\mu_{t}^{N}-\mu_{0}^{N}, G\right\rangle\right|>\varepsilon\right)=-\infty.
\end{equation}
\end{lemma}

\,

We first explain why the above lemma implies exponential tightness. For any $m \in \mathbb{N},\, k \in \mathbb{Z} $ and any $\delta, A > 0$, define
\begin{align*}
B_{k,A}  = \left\{ \sup_{0 \leq t \leq T} \left|  \mu_t (\theta_k ) \right| \leq A \right\}, \quad
B_{k,m,\delta}  = \left\{ \sup_{0 \leq |t -s| \leq \delta} \left|  (\mu_t - \mu_s) (\theta_k)  \right| \leq \frac{1}{m} \right\}.
\end{align*}
Then by Lemma \ref{exptight}, for any $n > 0$, there exist $A = A (n,k)$ and $\delta = \delta (m,k,n)$ such that
\begin{equation*}
\sup_{N \geq 1} Q^N_\rho \left[ B_{k,A}^c \right] < e^{-  (a_N^2/N) n k},\quad
\sup_{N \geq 1} Q^N_\rho \left[ B_{k,m,\delta}^c \right] < e^{-  (a_N^2/N) n k m}.
\end{equation*}
Let
$$
\mathcal{K}_n = \left\{ \bigcap_{k \geq 1}  B_{k,A(n,k)}  \right\} \bigcap
\left\{ \bigcap_{k,\,m \geq 1}  B_{k,m,\delta(m,k,n)} \right\}.
$$
It can be checked  that $\mathcal{K}_n$ is a compact set for each $n \geq 1$. Moreover, $Q^N_\rho \, [\mathcal{K}_n^c]$ is bounded by a multiple of $\exp \{ - (a_N^2 / N) n\}$. This proves the exponential tightness.

\,

\begin{proof}[Proof of Lemma \ref{exptight}]
We first prove \eqref{b1}. Since \eqref{r2} and \eqref{r4} are bounded,  we only need to show that
\begin{equation*}
\limsup _{A \rightarrow \infty}\, \limsup _{N \rightarrow \infty} \,\frac{N}{a_N^2} \log \mathbb{P}^{N}_{\rho} \,\left[ \sup_{0 \leqslant t \leqslant T} \left| \frac{N}{a^2_N} \log M_t^N (G) + \langle \mu^N_0,G\rangle + \int_0^t \langle \mu^N_s, \tilde{\Delta} G\rangle d s  \right| > A\right] = - \infty,
\end{equation*}
which is a consequence of
\begin{align}
\limsup _{A \rightarrow \infty} \,\limsup _{N \rightarrow \infty} \,\frac{N}{a_N^2} \log \mathbb{P}^{N}_{\rho}  \left[  \sup_{0 \leqslant t \leqslant T} \left| \frac{N}{a_N^2} \log M^N_t (G) \right| > A/3 \right] = -\infty,\label{b4}\\
\limsup _{A \rightarrow \infty}\, \limsup _{N \rightarrow \infty} \,\frac{N}{a_N^2} \log \mathbb{P}^{N}_{\rho} \left[ \frac{1}{a_N} \left| \sum_{x \in \mathbb{T}_N} \left( \eta_0 (x) - \rho \right)  G \left( x/N \right) \right| > A/3 \right] = -\infty,\label{b5}
\end{align}
and
\begin{equation}
\limsup _{A \rightarrow \infty} \,\limsup _{N \rightarrow \infty}\, \frac{N}{a_N^2} \log \mathbb{P}^{N}_{\rho}  \left[ \sup_{0 \leqslant t \leqslant T} \left| \int_0^t \langle \mu^N_s, \tilde{\Delta} G\rangle d s \right| > A/3 \right] = -\infty.\label{b6}
\end{equation}
Notice that \eqref{b6} follows from Lemma \ref{lem3}. To prove \eqref{b4}, without loss of generality,  we first remove the modulus since otherwise we can replace $G$ by $- G$.   Then
\begin{align*}
&\mathbb{P}^{N}_{\rho} \left[  \sup _{0\leqslant t \leqslant T}  \frac{N}{ a_N^2} \log M^N_t (G) > A/3 \right]
=  \mathbb{P}^{N}_{\rho} \left[  \sup _{0\leqslant t \leqslant T}  M^N_t (G) >  \exp \left\{ \frac{A a_N^2 }{3 N} \right\} \right]\\
& \leq  4 \exp \left\{ - \frac{A a_N^2 }{3 N} \right\} \mathbb{E}^N_\rho \left[\left( M_T^N (G) \right)^2\right] \leq 4 \exp \left\{ C \left(||G||_\infty^2 + ||G^\prime||_\infty^2\right) T - \frac{A a_N^2 }{3 N} \right\}.
\end{align*}
This proves \eqref{b4}. For \eqref{b5}, removing the modulus inside the probability as before and then by Chebyshev's inequality,
\begin{align*}
&\frac{N}{a_N^2} \log \mathbb{P}^{N}_{\rho} \left[ \frac{1}{a_N}  \sum_{x \in \mathbb{T}_N} \left( \eta_0 (x) - \rho \right)  G \left( x/N \right)  > A/3 \right] \\
& \leq - A/3 + \frac{N}{a_N^2} \log \mathbb{E}^{N}_{\rho} \left[ \exp \left\{\frac{a_N}{N}  \sum_{x \in \mathbb{T}_N} \left( \eta_0 (x) - \rho \right)  G \left( x/N \right)  \right\}  \right] \\
&= - A/3 + \frac{N}{a_N^2} \sum_{x \in \mathbb{T}_N} \log \left(  1 + \frac{C (\rho) a_N^2}{N^2} G(x/N)^2 + \mathcal{O}_G \left( \frac{a_N^3}{N^3} \right)  \right)\\
&\leq - A/3 + \frac{C (\rho)}{N} \sum_{x \in \mathbb{T}_N} G(x/N)^2 + \mathcal{O}_G \left( \frac{a_N}{N} \right).
\end{align*}
This proves \eqref{b5} by letting $N \rightarrow \infty$ and then $A \rightarrow \infty$.

Next we prove \eqref{u0}. Fix $A > 0$, which will converge to infinity after $\delta \rightarrow 0,\, N \rightarrow \infty$.  From Equations \eqref{r0}-\eqref{r5} with $G$ replaced by $A G$,  we only need to prove \eqref{u0} for the following four terms:
$$
\frac{N}{A a_N^2} \log M^N_t (A G), \quad   \int_0^t \langle \mu^N_s, \tilde{\Delta} G\rangle ds,
$$
$$ A \int_0^t \frac{1}{2}  (\eta_s (-1) - \eta_s (0))^2 \left( G \left(\frac{0}{N}\right) - G \left(\frac{-1}{N}\right) \right)^2 ds,
$$
and
$$
A \int_0^t \frac{1}{2 N} \sum_{x \in \mathbb{T}_N,\atop x \neq -1}   (\eta_s (x) - \eta_s (x+1))^2 \left(\nabla_N G \left(\frac{x}{N}\right)\right)^2 d s.
$$
Notice that the proof of \eqref{b4} also applies to the martingale term. The second one  follows from Lemma \ref{lem3}. For the last two terms, notice that they are both bounded by $C (G) \delta A$. The proof is complete.
\end{proof}

\section{Lower bound}\label{section of lower bound}

For  $f,\,g \in \mathscr{G}_0$, we define
\[
\langle f | g\rangle=2\rho(1-\rho)\left[\big(f(0) - f(1)\big)\big(g(0) - g(1)\big)+ \int_0^1 \partial_u f\,(u) \,\partial_u g \,(u)  \,d u\right].
\]
For $f,\,g\in \mathscr{G}$ and $0\leq t\leq T$, we define
\begin{align*}
\langle\langle f, g \rangle\rangle_t = \int_0^t \langle f_s | g_s\rangle \,ds.
\end{align*}
For simplicity, we write $\langle\langle f, g\rangle\rangle_T$ as $\langle\langle f, g\rangle\rangle$. To make $\langle\langle \cdot, \cdot\rangle\rangle$ an inner product, we write $f\simeq g$ if and only if $\langle\langle f-g, f-g \rangle\rangle=0$ and then define $\mathscr{H}$ as the Hilbert space which is the completion of $\mathscr{G}/_{\simeq}$.

For locally square integrable martingales $\{M_t\}_{t\geq 0}$ and $\{N_t\}_{t\geq 0}$, we use $\{\langle M, N\rangle_t\}_{t\geq 0}$ to denote the predictable quadratic-covariation process which is continuous and use $\{[M,N]_t\}_{t\geq 0}$ to denote the optional quadratic-covariation process which satisfies
\[
[M, N]_T= \lim_{\sup(t_{i+1}-t_i)\rightarrow 0}\,\sum_i\big(M_{t_{i+1}}-M_{t_i}\big)\big(N_{t_{i+1}}-N_{t_i}\big) \quad \text{in $L^2$},
\]
where the limit is over all   partitions $\{t_i\}$ of $[0, T]$. Note that $[M, N]=\langle M, N \rangle$ when $M$ and $N$ are continuous. For any $H\in C^{1,2}([0, T]\times \{0,1\}^{\mathbb{T}_N})$, by Dynkin's martingale formula,
\begin{equation}\label{martingale2}
\Lambda_t^N(H) :=H(t, \eta_t)-H(0, \eta_0)-\int_0^t (\mathcal{L}_N+\partial_s)H(s, \eta_s)\,ds
\end{equation}
 is a martingale and for any $H_1, H_2\in C^{1,2}([0, T]\times \{0,1\}^{\mathbb{T}_N})$,
\begin{align}\label{equ 5.0}
\langle \Lambda^N(H_1), \Lambda^N(H_2) \rangle_t =\int_0^t\mathcal{L}_N\big(H_1H_2\big)-H_1\mathcal{L}_NH_2
-H_2\mathcal{L}_NH_1 \,ds.
\end{align}

The following lemma  gives  clear expressions of $I_{dyn}$ and $I_{ini}$.

\,

\begin{lemma}\label{lemma 5.3}
(i) If $I_{dyn}\, (\mu)<+\infty$, then there exists $\psi\in \mathscr{H}$ such that $\ell_T(\mu, G)=\langle\langle G,\psi \rangle\rangle$ for any $G\in \mathscr{G}$ and $I_{dyn}\, (\mu)=\frac{1}{2}\langle\langle \psi,\psi \rangle\rangle$.\\
(ii) If $I_{ini}\,(\nu)<+\infty$ for  $\nu \in \mathscr{M}$, then there exists $\phi\in L^2[0,1]$ such that $\nu(G)=\langle \phi,G\rangle$ for any $G \in \mathscr{G}_0$ and
\[
I_{ini}\,(\nu)=\frac{\int_0^1\phi^2(u)\,du}{2\rho(1-\rho)}.
\]

\end{lemma}

\proof

The proofs of the two parts follow the same strategy, hence we only give the proof of $(i)$. According to the definition of $I_{dyn}$,
\[
I_{dyn} (\mu)=\sup_{G\in \mathscr{G}} \left\{\ell_{T}(\mu, G)-(1/2)\langle\langle G,G \rangle\rangle \right\}.
\]
If $\ell_{T}(\mu, G)\neq 0$ for some $G$ such that $\langle\langle G,G \rangle\rangle=0$, then
\[
I_{dyn}\,(\mu)\geq \sup_{c\in \mathbb{R}} \left\{ \ell_{T}(\mu, cG)-\frac{1}{2} \langle\langle cG, cG \rangle\rangle \right\}=\sup_{c\in \mathbb{R}} \left\{ c \, \ell_{T}(\mu, G) \right\}=+\infty,
\]
which is contradictory. Therefore, $\ell_{T}(\mu, \cdot)$ is well defined on $\mathscr{G}/_{\simeq}$. For $G\in \mathscr{G}/_{\simeq}$ such that $G\neq 0$,
$\ell_{T}(\mu, cG)-(1/2)\langle\langle cG, cG\rangle\rangle$ obtains maximum $\frac{\ell^2_T(\mu,G)}{2\langle\langle G,G\rangle\rangle}$ at $c=\frac{\ell_T(\mu, G)}{\langle\langle G,G\rangle\rangle}$. Therefore,
\[
I_{dyn} \,(\mu)=\sup_{G\in \mathscr{G}/_{\simeq}, G\neq 0}\, \frac{\ell^2_T(\mu,G)}{2\langle\langle G,G \rangle\rangle}.
\]
Since $I_{dyn} (\mu)<+\infty$, $\ell_T(\mu, \cdot)$ can be extended to a bounded linear function on $\mathscr{H}$. As a result, the existence of $\psi$ follows from Riesz's representation theorem and $I_{dyn} \,(\mu)=\frac{1}{2}\langle\langle\psi, \psi\rangle\rangle$ follows from Cauchy-Schwartz's inequality.
\qed

\,

For $\phi\in \mathscr{G}_0$ and sufficiently large $N$,  denote by $\nu_{_{N,\phi}}$ the product measure on $\{0,1\}^{\mathbb{T}_N}$ with marginals given by
\[
\nu_{_{N,\phi}}\,\{\eta: \eta(x)= 1\}=\rho+\frac{a_N}{N}\,\phi \left(\frac{x}{N} \right),\quad x \in \mathbb{T}_N,
\]
and by $\mathbb{P}^N_\phi$ the law of the process $\{\eta_t\}_{t\geq 0}$ with initial distribution $\nu_{_{N,\phi}}$.  For any $G\in \mathscr{G}$,  denote by $\widehat{\mathbb{P}}^N_{\phi, G}$ the probability measure on $D \left([0,T],\{0,1\}^{\mathbb{T}_N}\right)$ such that
\[
\frac{d\, \widehat{\mathbb{P}}^N_{\phi, G}}{d\,\mathbb{P}^N_\phi}=M_T^N(G).
\]

\,

\begin{lemma}\label{lemma 5.4}
For any $G\in \mathscr{G}$ and any $\phi\in \mathscr{G}_0$, $\{\mu^N_t\}_{0\leq t\leq T}$ converges in $\widehat{\mathbb{P}}^N_{\phi, G}$-probability to $\mu^G=\{\mu_t^G\}_{0\leq t\leq T}$ as $N \rightarrow \infty$, where $\mu^{G}$ is the unique element in $D\big([0,T], \mathscr{M}\big)$ such that
\begin{equation}\label{equ ODE for linear operator}
\begin{cases}
&\frac{d}{dt}\mu_t^G(h)=\mu_t^G(\tilde{\Delta}h)+\langle h|G_t\rangle, \\
&\mu_0^G(h)=\int_0^1\phi(u)h(u)\, du
\end{cases}
\end{equation}
for any $h\in \mathscr{G}_0$ and $0\leq t\leq T$.
\end{lemma}

\begin{remark}
Intuitively but not rigorously,  integrating by parts, $\mu^G_t$ given by  \eqref{equ ODE for linear operator} should be a signed measure such that $\mu^G_t(du)=\rho(t,u)\,du$, where $\rho(t,u)$ is the solution to the PDE
\[
\begin{cases}
&\partial_t\rho(t,u)=\tilde{\Delta}\rho(t,u)-2\rho(1-\rho)\tilde{\Delta}G_t(u),\\
&\rho(0,u)=\phi(u), ~0\leq u\leq 1, \\
&\rho(t, \cdot)\in \mathscr{G}_0, ~0\leq t\leq T.
\end{cases}
\]
However, as we have discussed in Remark \ref{remark 1.2}, we do not manage to prove the uniqueness or existence to this PDE. That's why we only consider $\mu^G$ as the solution to an equation on  $\mathscr{M}$ the space of  linear functionals on $\mathscr{G}_0$, the uniqueness and existence of which we can show rigorously.
\end{remark}

\,

To prove Lemma \ref{lemma 5.4}, we need some preparation. For  $h\in \mathscr{G}$, we write
$$
F_h(t, \eta)=\frac{1}{a_N}\sum_{x \in \mathbb{T}_N}\,(\eta(x)-\rho)\,h_t\left(\frac{x}{N}\right)
$$
and hence
\[
\Lambda^N_t(F_h)=\langle\mu^N_t, h_t\rangle-\langle\mu^N_0, h_0\rangle-\int_0^t (\mathcal{L}_N+\partial_s)\langle\mu_s^N, h_s\rangle \,ds.
\]

\,

\begin{lemma}\label{lemma 5.5}
For any $\phi\in \mathscr{G}_0$, $G, h\in \mathscr{G}$,
\begin{equation}\label{equ 5.3}
[\Lambda^N(F_h), \Lambda^N(F_h)]_T=o_{exp}(a_N^2)
\end{equation}
under both $\mathbb{P}^N_\phi$ and $\widehat{\mathbb{P}}^N_{\phi, G}$.
\end{lemma}

\proof

We first show that Equation \eqref{equ 5.3} holds under $\mathbb{P}^N_\phi$. According to the definition of $\Lambda^N_t(F_h)$,
\[
[\Lambda^N(F_h), \Lambda^N(F_h)]_T=\sum_{0\leq s\leq T}\big[\langle \mu_s^N, h_s\rangle-\langle\mu_{s-}^N, h_{s-}\rangle\big]^2.
\]
Recall that $\{Y_i (\cdot)\}_{i \in \mathbb{T}_N}$ are independent Poisson processes.  If $s$ is an event moment of $Y_i(\cdot)$, then
\[
\langle\mu_s^N, h_s\rangle-\langle\mu_{s-}^N, h_{s-}\rangle=\frac{1}{a_N}\left(h_{s-}\left(\frac{i+1}{N}\right)-h_{s-}\left(\frac{i}{N}\right)\right)\big(\eta_{s-}(i)-\eta_{s-}(i+1)\big).
\]
Consequently, let $C_h=\sup\limits_{0\leq t\leq T, \atop 0\leq u \leq 1}|h(t,u)|$ and $D_h=\sup\limits_{0\leq t\leq T, \atop 0\leq u\leq 1}|\partial_u h\,(t,u)|$, then
\[
[\Lambda^N(F_h), \Lambda^N(F_h)]_T\leq \frac{4}{a_N^2}\sum_{i\neq -1}\frac{Y_i(T)}{N^2}\,D_h^2+\frac{16}{a_N^2}\,C_h^2\,Y_{-1}(T)
\]
according to Lagrange's mean value theorem and the fact that there is at most one particle per site. By Chebyshev's inequality, for any $\theta>0$,
\begin{align*}
& \mathbb{P}^N_\phi\left([\Lambda^N(F_h), \Lambda^N(F_h)]_T\geq \epsilon\right) \leq \exp \left\{-\theta a_N^2\epsilon\right\} \,E\,\left[ \exp \left\{4\theta\sum_{i\neq -1}\frac{Y_i(T)}{N^2}D_h^2
+16\, \theta \,C_h^2\,Y_{-1}(T) \right\} \right] \\
&=\exp \left\{-\theta a_N^2\epsilon\right\} \left[\exp \left\{ N^2T \left(\exp \left\{\frac{4\theta D_h^2}{N^2}\right\}-1 \right) \right\}\right]^{N-1} \exp \left\{NT \left( \exp \{16\theta C_h^2\}-1 \right) \right\}.
\end{align*}
 Then
\[
\limsup_{N\rightarrow+\infty}\frac{1}{a_N^2}\log \mathbb{P}^N_\phi\big([\Lambda^N(F_h), \Lambda^N(F_h)]_T\geq \epsilon\big)\leq -\theta \epsilon.
\]
 This proves Equation \eqref{equ 5.3}  under $\mathbb{P}^N_\phi$ since $\theta$ is arbitrary.

Now we only need to show that Equation \eqref{equ 5.3} holds under $\widehat{\mathbb{P}}^N_{\phi, G}$. According to the definition of $\widehat{\mathbb{P}}^N_{\phi, G}$ and Cauchy-Schwartz inequality, for any $\epsilon>0$,
\[
\widehat{\mathbb{P}}^N_{\phi, G}\big([\Lambda^N(F_h), \Lambda^N(F_h)]_T\geq \epsilon\big)
\leq \sqrt{\mathbb{P}^N_\phi\big([\Lambda^N(F_h), \Lambda^N(F_h)]_T\geq \epsilon\big)}\sqrt{\mathbb{E}_\phi^N\left[\left(M_T^N(G)\right)^2\right]}.
\]
Recall the expressions of $M_t^N (G)$ given in \eqref{r0}-\eqref{r5}. It is not difficult to check that there exists a finite constant $C$ independent of $N$ such that $M_T^N(G)\leq e^{C\,a_N}$ for sufficiently large $N$. Therefore, Equation \eqref{equ 5.3} also holds under $\widehat{\mathbb{P}}^N_{\phi, G}$.
\qed

\,

\proof[Proof of Lemma \ref{lemma 5.4}]

The existence and uniqueness of Equation \eqref{equ ODE for linear operator} are given in the appendix.  It remains to show that $\mu^N$ converges weakly under  $\widehat{\mathbb{P}}^N_{\phi, G}$ to this unique solution $\mu^G$ as $N \rightarrow \infty$. To achieve this purpose, we need to investigate the martingale $\{M_t^N(G)\}_{t\geq 0}$  in \eqref{martingale1} and to utilize a generalized version of Girsanov's theorem introduced in \cite{Schuppen1974} by Schuppen and Wong.

Recall the definition of $\Lambda_t^N(f)$  in \eqref{martingale2} and that for any $G \in \mathscr{G}$,
$$
 f_G (t,\eta) = \exp \left\{\frac{a_N}{N}\sum_{i \in \mathbb{T}_N} (\eta(i)-\rho)G_t\left(\frac{i}{N}\right)\right\}.
$$
According to  Ito's formula,
\[
dM_t^N(G)=f_G (0, \eta_0)^{-1} \exp \left\{-\int_0^t\frac{(\partial_s+\mathcal{L}_N)f_G}{f_G} (s,\eta_s) \,ds\right\} d\Lambda_t^N(f_G).
\]
For any $t\geq 0$, let
\[
\widetilde{\Lambda}_t^N(f_G)=\int_0^t \frac{1}{f_G (s-, \eta_{s-})}\,d\Lambda_s^N(f_G),
\]
then
\begin{equation}\label{equ 5.5}
dM_t^N(G)=M_{t-}^N(G)\,d\widetilde{\Lambda}_t^N(f_G).
\end{equation}

\,

For any local martingale $\{M_t\}_{t\geq 0}$ under $\mathbb{P}_\phi^N$, let
$$
\widehat{M}_t=M_t- \left\langle M, \widetilde{\Lambda}^N(f_G) \right\rangle_t.
$$
By Equation \eqref{equ 5.5} and the generalized version of Girsanov's theorem \cite{Schuppen1974}, $\{\widehat{M}_t\}_{t\geq 0}$ is a local martingale under $\widehat{\mathbb{P}}^N_{\phi, G}$ and $[\widehat{M}, \widehat{M}] =[M, M]$ under both $\mathbb{P}_\phi^N$ and $\widehat{\mathbb{P}}^N_{\phi, G}$.  Therefore, for any $h\in \mathscr{G}$,
\begin{align*}
\langle \mu_t^N, h_t\rangle=&\langle \mu_0^N, h_0\rangle+\int_0^t (\partial_s+\mathcal{L}_N)\langle\mu_s^N, h_s\rangle ds+\widehat{\Lambda_t^N(F_h)}+\langle\Lambda^N(F_h), \widetilde{\Lambda}^N(f_G)\rangle_t,
\end{align*}
where $\left\{\widehat{\Lambda_t^N(F_h)}\right\}_{t\geq 0}$ is a local martingale under $\widehat{\mathbb{P}}^N_{\phi, G}$ with
\[
\left[\widehat{\Lambda^N(F_h)}, \widehat{\Lambda^N(F_h)}\right]=\left[\Lambda^N(F_h), \Lambda^N(F_h)\right].
\]
Then, by Lemma \ref{lemma 5.5} and Doob's inequality, $\widehat{\Lambda_t^N(F_h)}=o_p(1)$ under $\widehat{\mathbb{P}}^N_{\phi, G}$ and hence
\[
\langle \mu_t^N, h_t\rangle=\langle \mu_0^N, h_0\rangle+\int_0^t (\partial_s+\mathcal{L}_N)\langle\mu_s^N, h_s\rangle ds+o_p(1)+\langle \Lambda^N(F_h), \widetilde{\Lambda}^N(f_G) \rangle_t
\]
under $\widehat{\mathbb{P}}^N_{\phi, G}$.

\,

Next we calculate $\langle \Lambda^N(F_h), \widetilde{\Lambda}^N(f_G) \rangle_t$. According to the definition of $\widetilde{\Lambda}^N_t(f_G)$ and  \eqref{equ 5.0},
\[
d \left\langle \Lambda^N(F_h), \widetilde{\Lambda}^N(f_G) \right\rangle_t=\frac{1}{f_G(t, \eta_t)}d \left\langle \Lambda^N(F_h), \Lambda^N(f_G) \right\rangle_t,
\]
where
\begin{align*}
d \langle \Lambda^N(F_h), \Lambda^N(f_G) \rangle _t=\big(\mathcal{L}_N\big(F_hf_G\big)-f_G\mathcal{L}_NF_h-F_h\mathcal{L}_Nf_G\big)\,dt.
\end{align*}
By direct calculations,
\begin{align*}
\frac{1}{f_G}\Big(\mathcal{L}_N\big(F_hf_G\big)-f_G\mathcal{L}_NF_h-F_h\mathcal{L}_Nf_G\Big)={\rm \uppercase\expandafter{\romannumeral1}}_N
+{\rm \uppercase\expandafter{\romannumeral2}}_N,
\end{align*}
where
\[
\begin{aligned}
&{\rm \uppercase\expandafter{\romannumeral1}}_N
=\frac{N^2}{a_N}\sum_{i\neq-1}\big(\eta_t(i+1)-\eta_t(i)\big)\left(h_t\left(\frac{i}{N}\right)-h_t\left(\frac{i+1}{N}\right)\right) \\
& \quad \quad \quad \times \left(\exp \left\{\frac{a_N}{N}\big(\eta_t(i+1)-\eta_t(i)\big)\left(G_t\left(\frac{i}{N}\right)-G_t\left(\frac{i+1}{N}\right)\right)\right\}-1\right)
\end{aligned}
\]
and
\[
\begin{aligned}
&{\rm \uppercase\expandafter{\romannumeral2}}_N
=\frac{N}{a_N}\big(\eta_t(0)-\eta_t (-1)\big) \left(h_t\left(\frac{-1}{N}\right)-h_t\left(\frac{0}{N}\right)\right) \\
& \quad \quad \quad \times \left(\exp \left\{\frac{a_N}{N}\big(\eta_t(0)-\eta_t(-1)\big)\left(G_t\left(\frac{-1}{N}\right)-G_t\left(\frac{0}{N}\right)\right)\right\}-1\right).
\end{aligned}
\]
By Taylor's expansion formula up to second order,
\[
{\rm \uppercase\expandafter{\romannumeral1}}_N=\frac{1}{N}\sum_{i\neq -1}\big(\eta_t(i)-\eta_t(i+1)\big)^2\,\partial_uh_t\left(\frac{i}{N}\right)\partial_uG_t\left(\frac{i}{N}\right)+o(1)
\]
and
\[
{\rm \uppercase\expandafter{\romannumeral2}}_N=\big(\eta_t(-1)-\eta_t(0)\big)^2\big(h_t(0)-h_t(1)\big)\big(G_t(0)-G_t(1)\big)+o(1).
\]
Since
$
\big(\eta_t(i)-\eta_t(i+1)\big)^2=\eta_t(i)(1-\eta_t(i+1))+\eta_t(i+1)(1-\eta_t(i))
$,
Lemmas \ref{lem1} and \ref{lem2} control the errors when we replace $\big(\eta_t(i)-\eta_t(i+1)\big)^2$ by $2\rho(1-\rho)$ in ${\rm \uppercase\expandafter{\romannumeral1}}_N$ and ${\rm \uppercase\expandafter{\romannumeral2}}_N$.  To be precise, under $\mathbb{P}_\rho^N$,
\begin{equation}\label{equ 5.6}
\int_0^T {\rm \uppercase\expandafter{\romannumeral1}}_N \,dt
=2\rho(1-\rho)\int_0^T\,\int_0^1\partial_uh_t\,(u)\,\partial_uG_t(u)\,du\,dt+o(1)+o_{exp}(a_N)
\end{equation}
and
\begin{equation}\label{equ 5.7}
\int_0^T {\rm \uppercase\expandafter{\romannumeral2}}_N\,dt=2\rho(1-\rho)\int_0^T \big(h_t(0)-h_t(1)\big)\big(G_t(0)-G_t(1)\big)dt+o(1)+o_{exp}(a_N).
\end{equation}
By Taylor's expansion formula, it is not difficult to show that there exists a finite constant $C$ independent of $N$ such that
$\frac{d\,\nu_{_{N,\phi}}}{d\,\nu_\rho}\leq e^{C a_N}$ for sufficiently large $N$. Therefore, $\frac{d\widehat{\mathbb{P}}^N_{\phi, G}}{d\mathbb{P}^N_\rho}\leq e^{C a_N}$ for  large $N$. By Cauchy-Schwartz inequality, Equations \eqref{equ 5.6} and \eqref{equ 5.7} also hold under $\widehat{\mathbb{P}}^N_{\phi, G}$. As a result, under $\widehat{\mathbb{P}}^N_{\phi, G}$,
\begin{align*}
&\langle \mu_t^N, h_t\rangle
=o_p(1)+\langle \mu_0^N, h_0\rangle+\int_0^t (\partial_s+\mathcal{L}_N)\langle\mu_s^N, h_s\rangle ds\\
&+2\rho(1-\rho)\Bigg(\int_0^t \big(h_s(0)-h_s(1)\big)\big(G_s(0)-G_s(1)\big)ds
+\int_0^t\,\int_0^1\,\partial_uh_s\,(u)\,\partial_uG_s\,(u)\,du\,ds\Bigg).
\end{align*}

\,

Now we calculate $(\partial_s+\mathcal{L}_N)\langle\mu_s^N, h_s\rangle$. By direct calculations,
\begin{align*}
\mathcal{L}_N\langle\mu_s^N, h_s\rangle=&\frac{N^2}{a_N}\sum_{i=0}^{N-1}(\eta_s(i)-\rho)
\left(h_s\left(\frac{i+1}{N}\right)+h_s\left(\frac{i-1}{N}\right)-2h_s\left(\frac{i}{N}\right)\right)\\
&+\frac{N-N^2}{a_N}\big(\eta_s(0)-\eta_s(-1)\big) \left(h_s\left(\frac{-1}{N}\right)-h_s(0)\right)={\rm \uppercase\expandafter{\romannumeral3}}_N+{\rm \uppercase\expandafter{\romannumeral4}}_N,
\end{align*}
where
\[
{\rm \uppercase\expandafter{\romannumeral3}}_N=\frac{N^2}{a_N}\sum_{i\neq 0,-1}\big(\eta_s(i)-\rho\big)\left(h_s\left(\frac{i+1}{N}\right)+h_s\left(\frac{i-1}{N}\right)-2h_s\left(\frac{i}{N}\right)\right)
\]
and
\begin{align*}
{\rm \uppercase\expandafter{\romannumeral4}}_N=&\frac{N}{a_N}(\eta_s(0)-\eta_s(-1))\left(h_s\left(\frac{-1}{N}\right)-h_s(0)\right)
+\frac{N^2}{a_N}(\eta_s(0)-\rho)\left(h_s\left(\frac{1}{N}\right)-h_s(0)\right)\\
&+\frac{N^2}{a_N}(\eta_s(-1)-\rho)\left(h_s\left(\frac{-2}{N}\right)-h_s(-1)\right).
\end{align*}
By Taylor's expansion formula up to third order,
 $$
 {\rm \uppercase\expandafter{\romannumeral3}}_N=\langle\mu_s^N, \tilde{\Delta}h_s\rangle+o(1).
 $$
Since $h\in \mathscr{G}$, it is not difficult to check that ${\rm \uppercase\expandafter{\romannumeral4}}_N=o(1)$.

\,

In conclusion, we have shown that under $\widehat{\mathbb{P}}^N_{\phi, G}$,
\begin{align*}
&\langle \mu_t^N, h_t\rangle
=o_p(1)+\langle \mu_0^N, h_0\rangle+\int_0^t \langle \mu_s^N, (\partial_s+\tilde{\Delta})h_s\rangle ds\\
&+2\rho(1-\rho)\Bigg(\int_0^t \big(h_s(0)-h_s(1)\big)\big(G_s(0)-G_s(1)\big)dt
+\int_0^t\,\int_0^1\,\partial_uh_s(u)\,\partial_uG_s(u)\,du\,ds\Bigg)\\
&=o_p(1)+\langle \mu_0^N, h_0\rangle+\int_0^t \langle \mu_s^N, (\partial_s+\tilde{\Delta})h_s\rangle ds+\int_0^t\langle h_s|G_s\rangle ds.
\end{align*}
Specially, when $h\in \mathscr{G}_0$,
\[
\langle \mu_t^N, h\rangle
=o_p(1)+\langle \mu_0^N, h\rangle+\int_0^t \langle \mu_s^N, \tilde{\Delta}h\rangle \,ds+ \int_0^t\langle h|G_s\rangle \,ds
\]
for all $0\leq t\leq T$. Note that although the $o_p(1)$ term in the above equation is given for each $t$, it is easy to check that this $o_p(1)$ term can be chosen uniformly for $0\leq t\leq T$.
Since
\[
\mu_t^G(h)=\int_0^1 \phi(u)h(u)\,du+\int_0^t \mu_s^G(\tilde{\Delta}h)\,ds+\int_0^t\langle h|G_s\rangle \,ds,
\]
by Grownwall's inequality,
\[
\big|\langle \mu_t^N, \theta_m\rangle-\mu^G(\theta_m)\big|\leq \Big(o_p(1)+\big|\int_0^1 \phi(x)\theta_m(x)dx-\langle \mu_0^N, \theta_m\rangle\big|\Big)e^{|e_m|t}
\]
for all $0\leq t\leq T$ and $m\geq 1$.

Therefore, to show that $\mu^N$ converges in  $\widehat{\mathbb{P}}^N_{\phi, G}$-probability to $\mu^G$ in $D\left([0,T],\mathscr{M}\right)$, we only need to show that
\begin{equation}\label{equ 5.8}
\langle \mu_0^N, h\rangle=\int_0^1 \phi(u)h(u)\,du+o_p(1)
\end{equation}
under $\widehat{\mathbb{P}}^N_{\phi, G}$ for any $h\in \mathscr{G}_0$. According to the definition of $\nu_{_{N, \phi}}$ and Chebyshev's inequality, it is easy to check that Equation \eqref{equ 5.8} holds under $\mathbb{P}_\phi^N$. Since $M_0^N(G)=1$, $\mu_0^N$ has the same distribution under $\mathbb{P}_\phi^N$ and $\widehat{\mathbb{P}}^N_{\phi, G}$. This finishes the proof.
\qed

\,

\proof[Proof of the lower bound]

If $\inf_{\mu \in O}  I (\mu) =+\infty$, then Equation \eqref{equ lower bound} holds trivially. So we only need to deal with the case where
$\inf_{\mu \in O}  I (\mu)<+\infty$. For given $\epsilon>0$, there exists $\mu^\epsilon\in O$ such that
\[
I_{ini}(\mu^\epsilon_0)+I_{dyn} (\mu^\epsilon)\leq \inf_{\mu \in O}\, I (\mu)+\epsilon.
\]
By Lemma \ref{lemma 5.3}, there exists $\phi^\epsilon\in L^2[0,1]$ and $\psi^\epsilon\in \mathscr{H}$ such that
$$
\mu_0^\epsilon\,(G)=\langle\phi^\epsilon,G\rangle, \,\forall G \in \mathscr{G}_0, \quad  I_{ini}\,(\mu_0^\epsilon)=\frac{\int_0^1(\phi^\epsilon(u))^2\,du}{2\rho(1-\rho)},
$$
and
$$
\ell_{T}(\mu^\epsilon, G)=\langle\langle G, \psi^\epsilon \rangle\rangle,\,\forall G\in \mathscr{G}, \quad
I(\mu^\epsilon)=\frac{1}{2}\langle\langle\psi^\epsilon, \psi^\epsilon\rangle\rangle.
$$
Let $G\in \mathscr{G}$ such that $G_t=b_th$ for some $h\in \mathscr{G}_0$ and $b\in C^1[0,T]$. By the above formula and \eqref{l_T},
\[
b_T\mu_T^\epsilon(h)-b_0\mu_0^\epsilon(h)-\int_0^T b^\prime(s)\mu^\epsilon_s(h)ds=\int_0^Tb (s)\Big(\mu_s^\epsilon(\tilde{\Delta}h)+\langle h|\psi_s^\epsilon\rangle\Big) ds.
\]
Since $b$ is arbitrary, according to the formula of integration by parts, $\{\mu_t^\epsilon(h)\}_{0\leq t\leq T}$ is absolutely continuous and
\begin{equation}\label{equ 5.10}
\begin{cases}
&\frac{d}{dt}\mu_t^\epsilon(h)=\mu_t^{\epsilon}(\tilde{\Delta}h)+\langle h|\psi_t^\epsilon\rangle,\\
&\mu_0^\epsilon(h)=\int_0^1\phi^\epsilon(u)h(u) \,du
\end{cases}
\end{equation}
for any $h\in \mathscr{G}_0$.

\,

Since $\mathscr{G}_0$ is dense in $L^2[0,1]$ by Lemma \ref{lemma 1.1.1 topology} and $\mathscr{G}$ is dense in $\mathscr{H}$, there exist $\phi_n\in \mathscr{G}_0$ and $\psi_n\in \mathscr{G}$ such that
$\phi_n$ converges to $\phi^\epsilon$ in $L^2[0,1]$ and $\psi_n$ converges to $\psi^\epsilon$ in $\mathscr{H}$ as $n\rightarrow\infty$. Let $\mu_n\in D\big([0,T], \mathscr{M}\big)$ such that $\mu_{n,0}(G)=\langle \phi_n,G\rangle$ for any $G \in \mathscr{G}_0$, and $\ell_{T}(\mu_n, G)=\langle \langle G, \psi_n\rangle \rangle$ for any $G\in \mathscr{G}$.  According to an analysis similar with that leading to Equation \eqref{equ 5.10}, $\mu_n$ is the solution to the Equation
\begin{equation}\label{equ 5.11}
\begin{cases}
&\frac{d}{dt}\mu_{n,t}(h)=\mu_{n,t}(\tilde{\Delta}h)+\langle h|\psi_{n,t}\rangle, \\
&\mu_{n,0}(h)=\int_0^1 \phi_n(u)h(u) \,du
\end{cases}
\end{equation}
for any $h\in \mathscr{G}_0$.
By Lemma \ref{lemma 5.3}, $$
I_{ini}(\mu_{n,0})=\frac{\int_0^1(\phi_n(u))^2\,du}{2\rho(1-\rho)}
$$
and
$$
I_{dyn}\,(\mu_n)=\frac{1}{2} \langle\langle \psi_n, \psi_n \rangle\rangle =\ell_T(\mu_n, \psi_n)-\frac{1}{2} \langle\langle \psi_n, \psi_n \rangle\rangle.
$$

By \eqref{equ 5.10}, \eqref{equ 5.11} and Grownwall's inequality, for any $0\leq t\leq T$ and any integer $k$,
\[
\big|\mu_{n,t}(\theta_k)-\mu_t^\epsilon(\theta_k)\big|\leq \Big|\int_0^1 (\phi^\epsilon(u)-\phi_n(u))\theta_k(u)\,du+\langle \langle \theta_k, \psi^\epsilon-\psi^n\rangle\rangle\Big|e^{|e_k|t}.
\]
Consequently, $\mu_n$ converges to $\mu^\epsilon$ in $D \left([0,T],\mathscr{M}\right)$ and
\[
\lim_{n\rightarrow+\infty}\big(I_{dyn}\,(\mu_n)+I_{ini}(\mu_{n,0})\big)=I_{dyn}\,(\mu^\epsilon)+I_{ini}(\mu_0^\epsilon).
\]
Hence, there exists $m\geq 1$ such that $\mu_m\in O$ and
\[
I_{dyn}\,(\mu_m)+I_{ini}(\mu_{m,0})\leq I_{dyn}\,(\mu^\epsilon)+I_{ini}(\mu_0^\epsilon)+\epsilon.
\]
Let $D_{\epsilon}=\big\{\mu:~|\ell_T(\mu, \psi_m)-\ell_T(\mu_m, \psi_m)|<\epsilon\big\}\bigcap O$, then by Lemma \ref{lemma 5.4} and Equation \eqref{equ 5.11}, $\mu^N$ converges in $\widehat{\mathbb{P}}^N_{\phi_m, \psi_m}$-probability to $\mu_m$ as $N\rightarrow+\infty$ and hence
\begin{equation*}
\lim_{N\rightarrow+\infty}\widehat{\mathbb{P}}^N_{\phi_m, \psi_m}\big(\mu^N\in D_{\epsilon}\big)=1.
\end{equation*}
According to the expression of $M_T^N(G)$ given in Equation \eqref{r5} and Lemmas \ref{lem1} and \ref{lem2},
\[
M_T^N(\psi_m)=\exp \left\{\frac{a_N^2}{N}\Big(\ell_T(\mu^N, \psi_m)-\frac{1}{2} \langle\langle \psi_m, \psi_m \rangle\rangle +o(1)+\widehat{\varepsilon}_N\Big)\right\},
\]
where $\widehat{\varepsilon}_N=o_{exp}(a_N)$ under $\mathbb{P}_\rho^N$.
As we have shown above, $\frac{d\widehat{\mathbb{P}}^N_{\phi_m, \psi_m}}{d\mathbb{P}_\rho^N}\leq e^{C a_N}$ for sufficiently large $N$, hence $\widehat{\varepsilon}_N=o_{exp}(a_N)$ under $\widehat{\mathbb{P}}^N_{\phi_m, \psi_m}$.

According to the definition of $\nu_{_{N,\phi_m}}$, Chebyshev's inequality and Taylor's expansion formula up to second order, it is not difficult to show that
\[
\frac{d\mathbb{P}_\rho^N}{d\mathbb{P}^N_{\phi_m}}=\exp\left\{-\frac{a_N^2}{N}\left(\frac{\int_0^1 \phi_m^2(u)\,du}{2\rho(1-\rho)}+\widetilde{\varepsilon}_N\right)\right\}
=\exp\left\{-\frac{a_N^2}{N}\left(I_{ini}(\mu_{m,0})+\widetilde{\varepsilon}_N\right)\right\},
\]
where $\widetilde{\varepsilon}_N=o_p(1)$ under $\widehat{\mathbb{P}}^N_{\phi_m, \psi_m}$. Consequently, let
\[
\widehat{D}_{N, \epsilon}=\{\mu^N\in D_{\epsilon}\}\bigcap \{|\widehat{\varepsilon}_N|<\epsilon, |\widetilde{\varepsilon}_N|<\epsilon\},
\]
then
\begin{equation}\label{equ 5.9}
\lim_{N\rightarrow+\infty}\widehat{\mathbb{P}}^N_{\phi_m, \psi_m}\big(\widehat{D}_{N, \epsilon}\big)=1.
\end{equation}
For sufficiently large $N$, on $\widehat{D}_{N, \epsilon}$,
\begin{align*}
&\frac{d\mathbb{P}_\rho^N}{d\widehat{\mathbb{P}}^N_{\phi_m, \psi_m}}=\frac{d\mathbb{P}_\rho^N}{d\mathbb{P}^N_{\phi_m}}\frac{d\mathbb{P}^N_{\phi_m}}{d\widehat{\mathbb{P}}^N_{\phi_m, \psi_m}}=\frac{d\mathbb{P}_\rho^N}{d\mathbb{P}^N_{\phi_m}}\frac{1}{M_T^N(\psi_m)} \\
&\geq \exp\Big\{-\frac{a_N^2}{N}\Big(I_{ini}(\mu_{m,0})+\ell_T(\mu_m, \psi_m)-\frac{1}{2}\langle\langle \psi_m, \psi_m \rangle\rangle+3\epsilon\Big)\Big\} \\
&=\exp\Big\{-\frac{a_N^2}{N}\Big(I_{ini}(\mu_{m,0})+I_{dyn}(\mu_m)+3\epsilon\Big)\Big\}\geq \exp\Big\{-\frac{a_N^2}{N}\Big(I_{ini}(\mu^\epsilon_0)+I_{dyn}(\mu^\epsilon)+4\epsilon\Big)\Big\}\\
&\geq \exp\Big\{-\frac{a_N^2}{N}\Big(\inf_{\mu\in O}\big(I_{ini}(\mu_0)+I_{dyn}(\mu)\big)+5\epsilon\Big)\Big\}.
\end{align*}
Therefore, by Equation \eqref{equ 5.9},
\begin{align*}
&\liminf_{N \rightarrow \infty} \frac{N}{a_N^2} \log Q^N_\rho\, [O]=\liminf_{n\rightarrow+\infty}\,\frac{N}{a_N^2}\log \mathbb{P}_\rho^N\big(\mu_N\in O\big)\geq \liminf_{n\rightarrow+\infty}\,\frac{N}{a_N^2}\log \mathbb{P}_\rho^N\big(\widehat{D}_{N, \epsilon}\big)\\
&=\liminf_{n\rightarrow+\infty}\,\frac{N}{a_N^2}\log \widehat{\mathbb{E}}^N_{\phi_m, \psi_m}\left[\frac{d\mathbb{P}_\rho^N}{d\widehat{\mathbb{P}}^N_{\phi_m, \psi_m}}\mathbf{1}_{\widehat{D}_{N, \epsilon}}\right]\geq -\inf_{\mu\in O}\big(I_{ini}(\mu_0)+I_{dyn}(\mu)\big)-5\epsilon.
\end{align*}
Since $\epsilon$ is arbitrary, the proof is complete.
\qed

\,

\appendix
\section{Appendix}

\subsection{Lemma \ref{lemma 1.1.1 topology}}

\proof[Proof of Lemma \ref{lemma 1.1.1 topology}]
We use $e_{-n}$ to denote $-(2\pi n)^2$ for $n\geq 0$ and use $e_n$ to denote $-k_n^2$ for $n\geq 1$. Let
\[
\widehat{\mathscr{G}}_0=\left\{G\in C^2[0,1]:~G^\prime(0)=G^\prime(1)=G(0)-G(1)\right\},
\]
then, by direct calculations, $\{e_n\}_{-\infty<n<+\infty}$ are all the eigenvalues of $\tilde{\Delta}$ limited on $\widehat{\mathscr{G}}_0$. Moreover, $\theta_{-n}(x)=\cos\big(2n\pi x\big)$ is the eigenvector with respect to $e_{-n}$ and $\theta_n(x)=\sin\big(k_n(x-\frac{1}{2})\big)$ is the eigenvector with respect to $e_n$.

According to the definition of the operator $\frac{d}{dx}\frac{d}{dW}$ introduced in \cite{Franco2009moderate}, $\tilde{\Delta}\Big|_{\hat{\mathscr{G}}_0}=\frac{d}{dx}\frac{d}{dW}$ for $W$ equal to Lebesgue measure plus the Dirac measure at $0$. Consequently, by \cite[Theorem 1]{Franco2009moderate}, $\{\theta_n\}_{-\infty<n<+\infty}$ is an orthogonal basis of $L^2[0,1]$.
\qed

\subsection{Equation \eqref{equ control of 2k moments}}

\proof[Proof of Equation \eqref{equ control of 2k moments}]

By Fubini's theorem,
\[
 {\rm E}_{\nu_\rho} \left[ \left(\eta^{M,{\rm R}} (0) - \rho \right)^{2k} \right]=
 \int_0^{+\infty} 2kt^{2k-1}{\rm P}_{\nu_\rho}\left(\left|\eta^{M,{\rm R}} (0) - \rho \right|\geq t\right) dt.
\]
For $0\leq t<1-\rho$ and $\theta\geq 0$, by Chebyshev's inequality,
\begin{align*}
{\rm P}_{\nu_\rho}\left(\eta^{M,{\rm R}} (0) - \rho\geq t\right)&\leq e^{-tM\theta}{\rm E}_{\nu_\rho}\left[e^{\theta M\left(\eta^{M,\rm{R}} (0) - \rho \right)}\right]=\left(e^{-\theta t}{\rm E}_{\nu_\rho}\left[e^{\theta\left(\eta (0) - \rho \right)}\right]\right)^M\\
&=\left(e^{-\theta(t+\rho)}\left[e^\theta\rho+1-\rho\right]\right)^M.
\end{align*}
Let $\theta=\log \frac{(t+\rho)(1-\rho)}{\rho(1-t-\rho)}$, then
\begin{equation}\label{equ A control moments}
{\rm P}_{\nu_\rho}\left(\eta^{M,{\rm R}} (0) - \rho\geq t\right)\leq e^{-M\mathscr{I}(t)}
\end{equation}
for $0\leq t< 1-\rho$, where
\begin{equation*}
\mathscr{I}(t)=-(1-t-\rho)\log(1-\rho)+(t+\rho)\log(t+\rho)+(1-t-\rho)\log(1-\rho-t)-(t+\rho)\log \rho.
\end{equation*}
We define $\mathscr{I}(1-\rho)=-\log \rho$. Note that $\lim_{t\uparrow1-\rho}\mathscr{I}(t)=-\log\rho$ and
\[
\rm{P}_{\nu_\rho}\left( \eta^{M,{\rm R}} (0) - \rho\geq 1-\rho\right)=\left(P_{\nu_\rho}\left( \eta(0) =1\right)\right)^M=\rho^M,
\]
hence Equation \eqref{equ A control moments} holds for $0\leq t\leq 1-\rho$ and $\mathscr{I}$ is continuous in $[0, 1-\rho]$. It is easy to check that
$\mathscr{I}(0)=0$ and $\mathscr{I}(t)>0$ for $t\in (0, 1-\rho]$. By L'Hospital's rule,
\[
\lim_{t\downarrow 0}\frac{\mathscr{I}(t)}{t^2}=\frac{1}{2\rho(1-\rho)}.
\]
Hence $\frac{\mathscr{I}(t)}{t^2}$ is continuous and strictly positive on $[0, 1-\rho]$. Let $J_1(\rho)=\inf_{0\leq t\leq 1-\rho}\frac{\mathscr{I}(t)}{t^2}$, which is strictly positive, then
\begin{equation}\label{equ A2}
{\rm P}_{\nu_\rho}\left(\eta^{M,{\rm R}} (0) - \rho\geq t\right)\leq e^{-MJ_1(\rho)t^2}
\end{equation}
for all $M\geq 1$ and any $t\in [0, 1-\rho]$ by Equation \eqref{equ A control moments}. Note that ${\rm P}_{\nu_\rho}\left( \eta^{M,{\rm R}} (0) - \rho\geq t\right)=0$ for $t>1-\rho$, hence Equation \eqref{equ A2} holds for all $M\geq 1$ and $t\geq 0$. A similar argument proves that there exists $J_2(\rho)>0$ such that
\[
{\rm P}_{\nu_\rho}\left(\eta^{M,{\rm R}} (0) - \rho\leq -t\right)\leq e^{-MJ_2(\rho)t^2}
\]
for all $M\geq 1$ and $t\geq 0$. Let $J(\rho)=\inf\{J_1(\rho),~J_2(\rho)\}$, then
\begin{align*}
{\rm E}_{\nu_\rho} \left[ \left(\eta^{M,{\rm R}} (0) - \rho \right)^{2k} \right]&=
 \int_0^{+\infty} 2kt^{2k-1}P_{\nu_\rho}\left(\left|\eta^{M,{\rm R}} (0) - \rho \right|\geq t\right) dt\\
 &\leq \int_0^{+\infty} 4kt^{2k-1}e^{-MJ(\rho)t^2}dt= \frac{2k!}{M^k\left(J(\rho)\right)^k}.
\end{align*}
Equation \eqref{equ control of 2k moments} follows  by taking $C(\rho)=\frac{2}{J(\rho)}$.
\qed

\subsection{Existence and Uniqueness of solution to Equation \eqref{equ ODE for linear operator}}

\proof[Proof of the existence]

We directly construct a solution to Equation \eqref{equ ODE for linear operator}. For $-\infty<n<+\infty$, let $\{x_t^n\}_{0\leq t\leq T}$ be the unique solution to the ODE
\[
\begin{cases}
&\frac{d}{dt}x_t^n=e_nx_t^n+\langle\theta_n|G_t\rangle,\\
&x_0^n=\int_0^1\phi(x)\theta_n(x)dx,
\end{cases}
\]
where $e_n$ is defined as in the proof of Lemma \ref{lemma 1.1.1 topology}. That is to say,
\[
x_t^n=e^{e_nt}\int_0^1 \phi(x)\theta_n(x)dx+\int_0^te^{e_n(t-s)}\langle\theta_n|G_s\rangle ds.
\]
For any $f=\sum_{-\infty<n<+\infty}C_n(f)\theta_n\in \mathscr{G}_0$ and $t\geq 0$, we define
\[
\mu_t^G(f)=\sum_{-\infty<n<+\infty}C_n(f)x_t^n.
\]
Note that the coefficients $\{C_n(f)\}_{-\infty<n<+\infty}$ are unique according to Lemma \ref{lemma 1.1.1 topology} and hence the definition of $\mu^G$ is reasonable. Since
\[
\mu_t^G(\theta_n)=x_t^n
\quad \text{and} \quad \mu_t^G(\tilde{\Delta}\theta_n)=\mu_t^G(e_n\theta_n)=e_nx_t^n,
\]
it is easy to check that $\mu^G$ is the solution to Equation \eqref{equ ODE for linear operator}.
\qed

\,

\proof[Proof of the uniqueness]

Assuming that $\mu$ and $\nu$ are both solutions to Equation \eqref{equ ODE for linear operator}, then
\[
|\mu_t(\theta_n)-\nu_t(\theta_n)|\leq |e_n|\int_0^t |\mu_s(\theta_n)-\nu_s(\theta_n)|ds.
\]
By Grownwall's inequality,
\[
|\mu_t(\theta_n)-\nu_t(\theta_n)|\leq 0e^{|e_n|t}=0
\]
for any $0\leq t\leq T$ and $n\geq 1$. Hence, $\mu=\nu$ and the proof is complete.
\qed

\,

\textbf{Acknowledgments.} The authors are grateful to the financial support from the National Natural Science Foundation of China with
grant numbers 11501542 and 11971038.

\bibliographystyle{plain}
\bibliography{zhaoreference}
	
\end{document}